\documentclass[aps,pre,notitlepage,nofootinbib,showkeys,tightenlines,12pt]{revtex4-1}

\usepackage{color}
\usepackage{amsmath,amsfonts,amssymb,amsthm}
\usepackage{mdwlist,booktabs}
\usepackage{graphicx,subfigure}

\graphicspath{{figs_arXiv/}}

\newcommand{\mathnotation}[2]{\newcommand{#1}{\ensuremath{#2}}}
\mathnotation{\hrods}{h_{\text{hom}}}
\mathnotation{\hflow}{h_{\text{flow}}}
\mathnotation{\V}{V}
\renewcommand{\H}{H}
\renewcommand{\L}{L}
\mathnotation{\M}{M}
\newcommand{\tild}[1]{\widetilde{#1}}
\mathnotation{\st}{\, | \, \,}
\mathnotation{\alphainv}{\alpha^{-1}}
\mathnotation{\betainv}{\beta^{-1}}

\newtheorem{theorem}{Theorem}

\newtheorem{lemma}[theorem]{Lemma}
\newtheorem{corollary}[theorem]{Corollary}

\begin{document}

\title{Topological entropy and secondary folding}

\author{Sarah Tumasz}
\affiliation{Department of Mathematics, University
  of Wisconsin, Madison, WI 53706, USA}
\author{Jean-Luc Thiffeault}
\email{jeanluc@math.wisc.edu}
\affiliation{Department of Mathematics, University
  of Wisconsin, Madison, WI 53706, USA}

\begin{abstract}
  A convenient measure of a map or flow's chaotic action is the
  topological entropy.  In many cases, the entropy has a homological
  origin: it is forced by the topology of the space.  For example, in
  simple toral maps, the topological entropy is exactly equal to the
  growth induced by the map on the fundamental group of the torus.
  However, in many situations the numerically-computed topological
  entropy is greater than the bound implied by this action.  We
  associate this gap between the bound and the true entropy with
  `secondary folding': material lines undergo folding which is not
  homologically forced.  We examine this phenomenon both for physical
  rod-stirring devices and toral linked twist maps, and show
  rigorously that for the latter secondary folds occur.
\end{abstract}

\maketitle

\section{Introduction}

In many industrial applications, fluid is stirred by moving rods to
achieve homogeneity.  Such a simple system also serves as a testbed
for ideas about mixing and transport barriers.  In the past few years,
a topological description of rod-stirring has emerged (see for
example~\cite{BoylandEtAl2000,ThiffeaultFinn2006,ThiffeaultPavliotis2008,FinnThiffeault2011}).
In this framework, the two-dimensional fluid lives
in a surface with holes, with the holes corresponding to moving rods
and fixed baffles.  This topological description applies to all
fluids, but is most useful for very viscous flows.

A consequence of the topological description is that we can compute
lower bounds on the \emph{topological entropy} of the system.  The
topological entropy is, roughly speaking, the growth rate of material
lines in a fluid~\cite{Adler1965,Newhouse1988}.  The lower bound on
the topological entropy is a consequence of continuity -- material
lines cannot cross the rods, so they must grow at least as fast as
dictated by the motion of the rods.

We can then ask about the sharpness of this lower bound: we measure
the actual growth rate of material lines in the flow, and compare it
to the lower bound.  In many systems, the `gap' between the lower
bound and the actual value is quite small -- on the order of a few
percent of the value.  In other cases, the gap is very large, so that
the observed value of the entropy is much larger than the lower bound.
The central question of this paper is: what accounts for the
difference in these two cases?  In both cases, much of the entropy is
`homological' and arises from obstructions in the domain -- the rods
and baffles.  This homological entropy depends only weakly on the
details of the physical system, and can be deduced without actually
solving the dynamical fluid equations.  However, in the cases with a
large gap it seems clear that there is another, dynamical source of
entropy, which can depend more strongly on system parameters, such as
the size of the stirring rods and the fluid being considered.  We
propose that the so called dynamical entropy is predominately due to
\emph{secondary folding}, where material lines exhibit folds that are
not directly associated with a topological obstacle.

When the difference between the bound and the measured entropy is
small, we refer to those bounds as \emph{sharp}.  We will investigate
the sharpness of the bound by considering simple three-rod stirring
devices.  The motion of the rods is limited to a sequence of clockwise
or counter-clockwise exchanges with one of their neighbors, following
circular paths~\cite{BoylandEtAl2000}.  These types of motions map
naturally to generators of the braid group~\cite{Birman1975}: a
clockwise exchange of the~$i$th and~$(i+1)$th rod is
denoted~$\sigma_i$, and a counter-clockwise exchange
by~$\sigma_i^{-1}$.  A sequence of~$\sigma_i^{\pm1}$ is called a
\emph{stirring protocol}.  Since the fluid we consider is highly
viscous, the stirring protocols define a periodic motion of the fluid
elements, obtained by solving Stokes' equation for incompressible
flow.  Figure~\ref{fig:sigmasline} shows some examples of stirring
protocols, as well as their action on typical material lines.

We will illustrate the gap between the bound and measured value using
two examples in Section~\ref{sec:observations}.  Because the disk with
three punctures (rods) has such an intimate connection with the torus,
we will examine toral maps in Section~\ref{sec:toral}.  We observe the
same gap phenomenon as in the fluid case, but because the map is a
much simpler system we are able to prove some explicit results about
toral maps: in particular, we demonstrate the presence of secondary
folds, or `kinks.'  Finally, in Section~\ref{sec:discussion} we
summarize our results and discuss future work.

\section{Some Observations}
\label{sec:observations}

As an illustration of the gap phenomenon in hydrodynamic flows, consider a stirring device
with three rods that initially lie in a line.  There are many ways to
move the rods and stir the fluid, but we are primarily interested in the
topological aspects of the motion.  As discussed in the introduction,
we focus on sequences of interchanges of neighboring rods, which we
label by generators of the braid group~\cite{BoylandEtAl2000,
  ThiffeaultFinn2006}.  Here, we consider two stirring protocols,
given in terms of generators by~$\sigma_1\sigma_2^{-1}$ and~$\sigma_1\sigma_2^5$, where generators are
read left to right.  The rod motions are depicted in the insets in
Figure \ref{fig:sigmasline}.  The exact details of how we move the
rods make no difference to the topology of the system, though it will
in general affect the measured growth rate of material lines.  The
speed of motion is irrelevant, since we are in the Stokes flow regime.
In the simulations below, we always use circular paths and constant
speed.
\begin{figure}
\begin{center}
\subfigure[\, $\sigma_1\sigma_2^{-1}$]{\label{fig:sigmaslineA}{ \includegraphics[width=0.3\columnwidth]{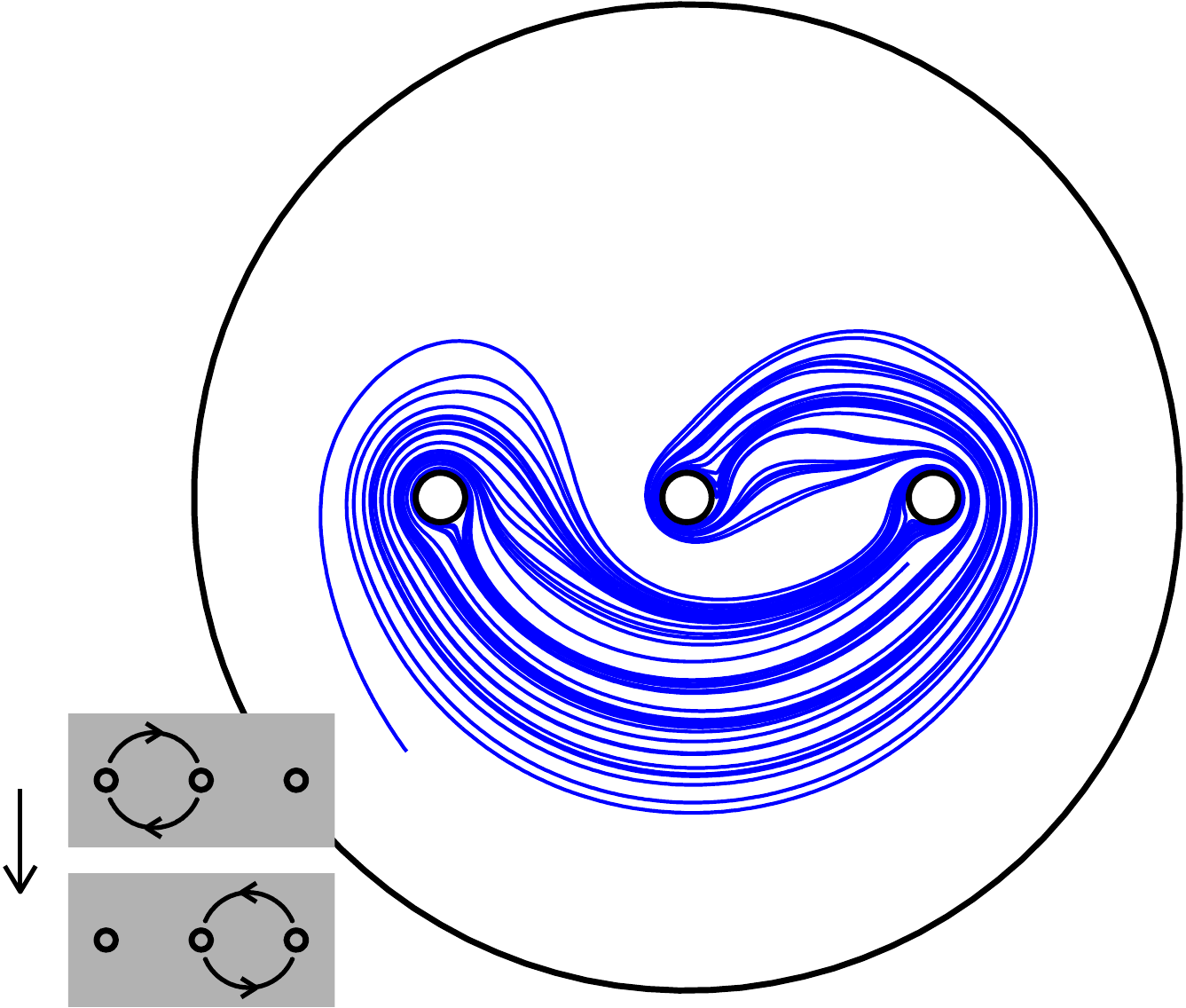}}}\hspace{1em}
\subfigure[\, $\sigma_1\sigma_2^5$]{\label{fig:sigmaslineC}{\includegraphics[width=0.3\columnwidth]{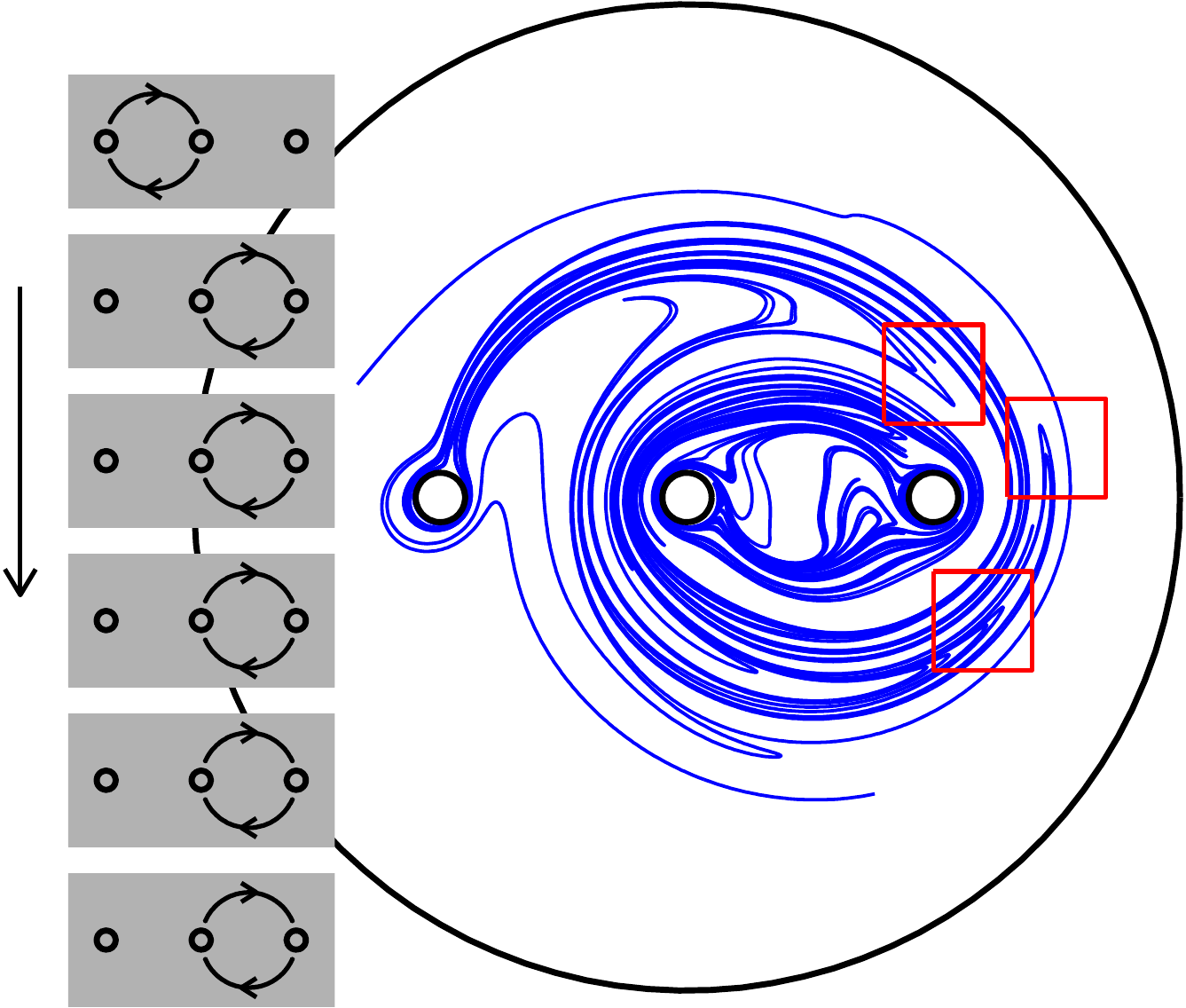} }}
\end{center}
\caption{The effect of two different stirring protocols on the same
  initial material line.  The insets show the motion of the rods.
  The images are after four and three periods, respectively. In (b) boxes
  highlight some instances of secondary folding.  There is no
  secondary folding visible in (a).}
\label{fig:sigmasline}
\end{figure}

\subsection{Computing $\hrods$ and $\hflow$}

Two types of topological entropy will be computed here.  The first,
called~$\hrods$, is the entropy coming from the homology, i.e.\ the
entropy due to the rod motion alone.  This is the lower bound alluded
to in the introduction: it is the growth rate of a hypothetical
`rubber band' caught on the rods, and is independent of both the type of
fluid being stirred and the specifics of the rod motion.  This entropy is computed directly from the braid
describing the rod motion.  The second entropy we compute, $\hflow$,
is obtained by directly measuring the growth rate of material lines in
the flow.  We have~$\hrods \le \hflow$, since material lines cannot
(asymptotically) grow slower than the minimum rate dictated by the
rods.  We now discuss how to compute~$\hrods$ and~$\hflow$.

Computing~$\hrods$ can be difficult when dealing with a large number
of rods or long sequences of generators, though rapid techniques have
been developed~\cite{Moussafir2006, Thiffeault2010}.  However, for
three rods the entropy is very easy to compute.  In that case, the
topological entropy is equal to logarithm of the spectral radius
(largest magnitude over all eigenvalues) of the Burau representation
matrix~\cite{Burau1936, Fried1986, Kolev1989, BoylandEtAl2003,
  BandBoyland2007}.%
\footnote{The (reduced) Burau representation is a representation of
  the braid group on~$n$ strings in terms of matrices of
  dimension~$n-1$.  The representation arises from an action on
  homology on the double-cover of the punctured disk.
  See~\cite{Birman1975, Birman2005} for more details.}
For braids of the form $\sigma_1^\ell\sigma_2^{-k}$ the Burau matrix is
\begin{equation}
[\sigma_1^{k}\sigma_2^{-\ell}] = \begin{bmatrix}
  1+k\ell & \ell \\ k & 1
\end{bmatrix}.
\label{eq:burau}
\end{equation}
If $\hrods > 0$, then, under repeated iterations, a rubber band caught on the rods will grow in length at an exponential rate.  In this case, the braid is called \emph{pseudo-Anosov}.  If $\hrods = 0$, the rubber band will not grow exponentially, and the braid is said to be \emph{finite-order}.  This terminology comes from the Thurston--Nielsen theory of classification of surface homeomorphisms~\cite{Fathi1979, Thurston1988}.  There is a third case in the classification, called \emph{reducible}, which is not relevant here.  Mixers with pseudo-Anosov stirring protocols are usually good at mixing; the exponential stretching and folding of material lines leads to a growth in the interface between solutes, which in turn allows diffusion or chemical reaction to act more rapidly.

Now for~$\hflow$: in theory it is computed by taking the supremum of
growth rates over all material lines in the fluid~\cite{Newhouse1988}.
In practice, a typical material line will eventually grow at a maximal
rate, as long as some part of it is in the ergodic component with the
fastest growth rate.  Thus, we can get a good measure of $\hflow$ for
each protocol by computing the rate at which a typical material line
stretches.  We do this by solving numerically for the motion of the
individual fluid elements making up a material line, taking care to
interpolate new points as the distance between neighbors becomes too
large.  

\subsection{Stirring Protocol $\sigma_1\sigma_2^{-1}$ vs. $\sigma_1\sigma_2^5$}

We compare two stirring protocols, given in braid form by
$\sigma_1\sigma_2^{-1}$ and $\sigma_1\sigma_2^5$.  Using the
matrix~\eqref{eq:burau}, we find that they both have $\hrods =
\log((3+\sqrt{5})/2) \simeq 0.962$.  Despite this, the effect on
material lines is quite different, as can be see by comparing
Figures~\ref{fig:sigmaslineA} and~\ref{fig:sigmaslineC}.  Notice that
for the $\sigma_1\sigma_2^{-1}$ protocol (Fig.~\ref{fig:sigmaslineA})
the material line forms very smooth and regular layers; the only large
visible folds are due to wrapping around the rods.  On the other hand,
for the~$\sigma_1\sigma_2^5$ protocol (Fig.~\ref{fig:sigmaslineC}),
the material line has extra folds that are not due to wrapping around
a rod, as highlighted by the boxes.  These are what we call `secondary
folds.'

From simulations\footnote{The computer simulation were performed using
  Matthew D.\ Finn's code for solving Stokes' equation.  This code
  uses complex variable methods to guarantee an accurate solution.}
we can measure $\hflow$ for each stirring protocol.  Figure
\ref{fig:sigmasentropy} shows the measured length of a material line
during several periods of the flow.  The length is plotted on a log
scale, and the (asymptotic) slope of the line is $\hflow$.  In this
way, we measure $\hflow=0.991$ for the $\sigma_1\sigma_2^{-1}$
protocol, and $\hflow=1.61$ for the $\sigma_1\sigma_2^5$ protocol.
For $\sigma_1\sigma_2^{-1}$, the rod entropy and the flow entropy
agree well (roughly a 3\% difference).  However, for
$\sigma_1\sigma_2^5$, the flow entropy is much larger than the rod
entropy ($\hrods$ accounts for roughly $61\%$ of $\hflow$).  We
hypothesize that the larger measured entropy is due to additional
stretching of the material line due to the secondary folds.
\begin{figure}
\begin{center}
\includegraphics[width=0.40\columnwidth]{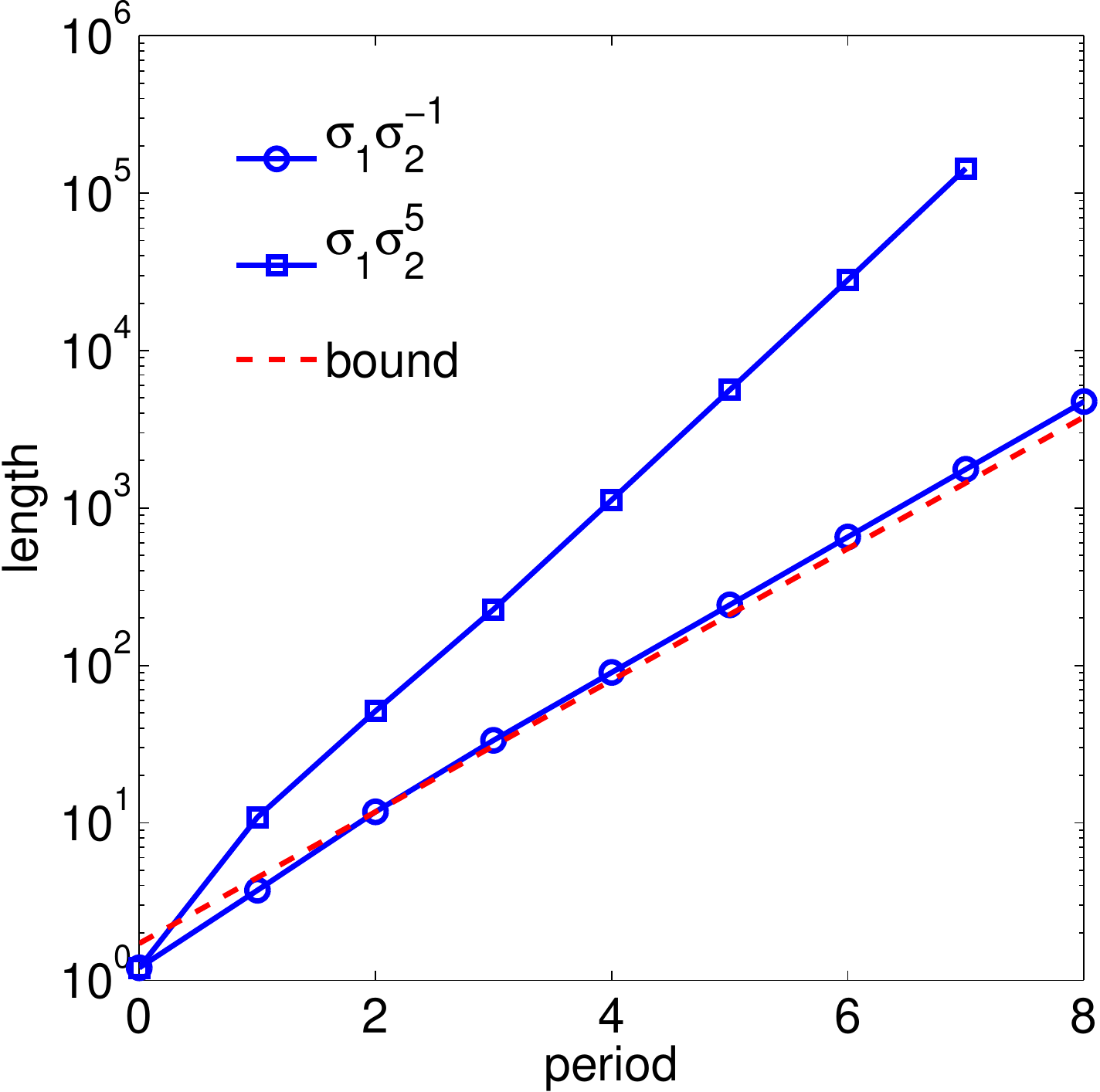} 
\end{center}
\caption{Growth of the length of a material line for the two stirring
  protocols.  Since they are both pseudo-Anosov systems, the slope,
  i.e.\ the growth rate tends to a constant -- namely the topological
  entropy.  The dashed line has slope equal to the lower
  bound,~$\hrods=0.962$.}
\label{fig:sigmasentropy}
\end{figure}

\section{Toral Linked Twist Map}
\label{sec:toral}

It is problematic to define secondary folding rigorously in a
practical context, which also makes it difficult to identify its
causes.  To make headway, we look at a simpler system -- a map instead
of a hydrodynamic flow -- which can be analyzed more thoroughly.  The
class of maps we'll examine are \emph{toral linked twist maps}.
Linked twist maps (or LTMs) have been studied
extensively~\cite{SturmanEtAl2006, Przytycki1983, Przytycki1986,
  SpringhamWiggins2010,Donnay1992, Devaney1978, Devaney1980,
  BurtonEaston1980}.  They are non-uniformly hyperbolic, and are
especially relevant here because of an intimate connection to the
three-rod stirring protocols above.  The connection arises through the
orientable double cover~\cite{BoylandEtAl2000,FarbMargalit2011}: the
torus can be regarded as a double cover of the disk with three
punctures.  The double cover branches at each of the punctures and at
the disk's outer boundary.  The interchange of rods given by
$\sigma_1$ and $\sigma_2$ in the three-rod system correspond to the
vertical and horizontal Dehn twists in the toral LTM.

The domain of the toral LTM is an L-shaped subset of the flat torus
formed by two overlapping orthogonal strips.  One full application of
the map consists of two consecutive linear shears.  The first shear
takes place in the vertical strip of width~$\alpha$, and the second
shear is in the horizontal strip of height~$\beta$ (Figure
\ref{fig:ttmshears}).  The integer parameters~$k$ and~$\ell$ are the
strengths of each shear.  In the limit~$\alpha=\beta=1$, toral LTMs
reduce to toral maps with uniform stretching.  These are often called
generalized Arnold cat maps: the cat map is the toral LTM with
$\alpha=\beta=k=\ell=1$.  Toral LTMs can also be defined more
generally with non-linear shears~\cite{SturmanEtAl2006}, but we do not
deal with those here.  If $k\ell>0$ we say that the system is
\emph{counter-rotating}; if $k\ell<0$ we say that it is
\emph{co-rotating}.  This naming convention is the opposite of the one
used by some other authors, for instance~\cite{SturmanEtAl2006}.  In
the present context, given the rod-stirring applications, it is more
natural to define co-rotating and counter-rotating as we do here.
\begin{figure}
\begin{center}
\includegraphics[width=.6\columnwidth]{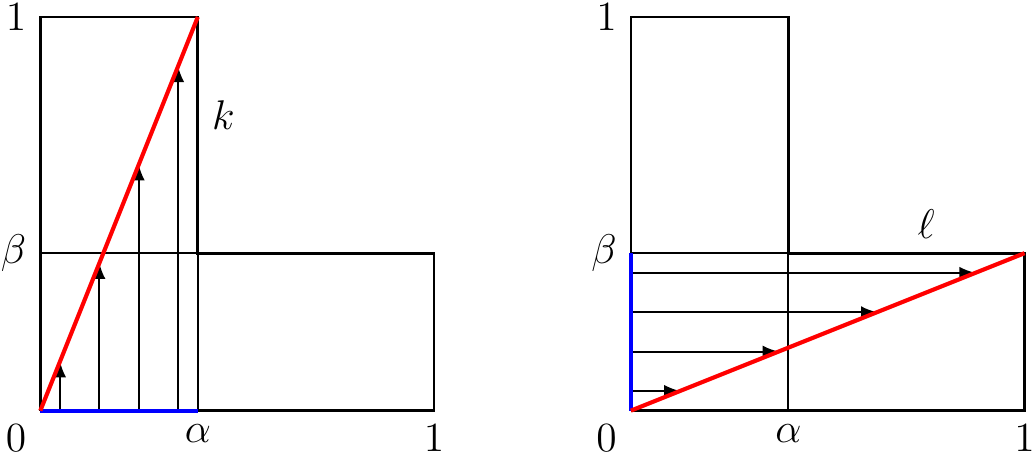} 
\end{center}
\caption{The toral LTM consists of two linear shears.  The first shear (left image) is in a vertical strip on the torus.  The left edge of the strip remains fixed, while the right edge is translated by $k$ units.  The second shear (right image) is in a horizontal strip.  The bottom edge remains fixed, while the top edge is translated by $\ell$ units.  Here, the shears are shown with $k=\ell=1$.}
\label{fig:ttmshears}
\end{figure}

Although we don't have a true fluid flow here, we can still compare
the topological entropy measured from the stretching rate of a
material line to a lower bound for the topological entropy of the map.
This lower bound arises from a semi-conjugacy between the toral LTM
and the generalized Arnold cat map with the same $k$ and
$\ell$~\cite{Franks1970}.  The generalized Arnold cat map acts on the
full torus with the same matrix given in (\ref{eq:burau}).  For $k$
and $\ell$ such that the matrix is hyperbolic, the topological entropy
is again the log of the magnitude of the largest eigenvalue.

As a parallel to our $\sigma_1\sigma_2^{-1}$ and $\sigma_1\sigma_2^5$
examples above, we look at toral LTMs with $k=1$, $\ell=1$
(counter-rotating) and $k=1$, $\ell=-5$ (co-rotating).  Figure
\ref{fig:ttmline} shows, for both the counter-rotating and co-rotating
maps, the image of an initial line segment after two iterations.  As
with the hydrodynamic flow examples, they both have a lower bound of
$\hrods= 0.962$.  Numerically measuring the stretching of the line
yields $\hflow = 0.962$ for the counter-rotating LTM and $\hflow=1.91$
for the co-rotating case.  As for the three-rod mixers, the entropy
agrees well with the lower bound for the counter-rotating example ($<1\%$
difference), but there is a large gap for the co-rotating example
($\hrods$ accounts for roughly $49.6\%$ of $\hflow$).  Looking again
at Figure \ref{fig:ttmline}, we notice visible secondary
folding in the co-rotating case (indicated by boxes), but none in the
counter-rotating case.  This holds in general for toral LTMs.  To
prove this, we take a closer look at the toral LTM map.
\begin{figure}
\center
\subfigure[\,counter-rotating]{\label{fig:ttmlinectrt}
  {\includegraphics[width=0.25\columnwidth]{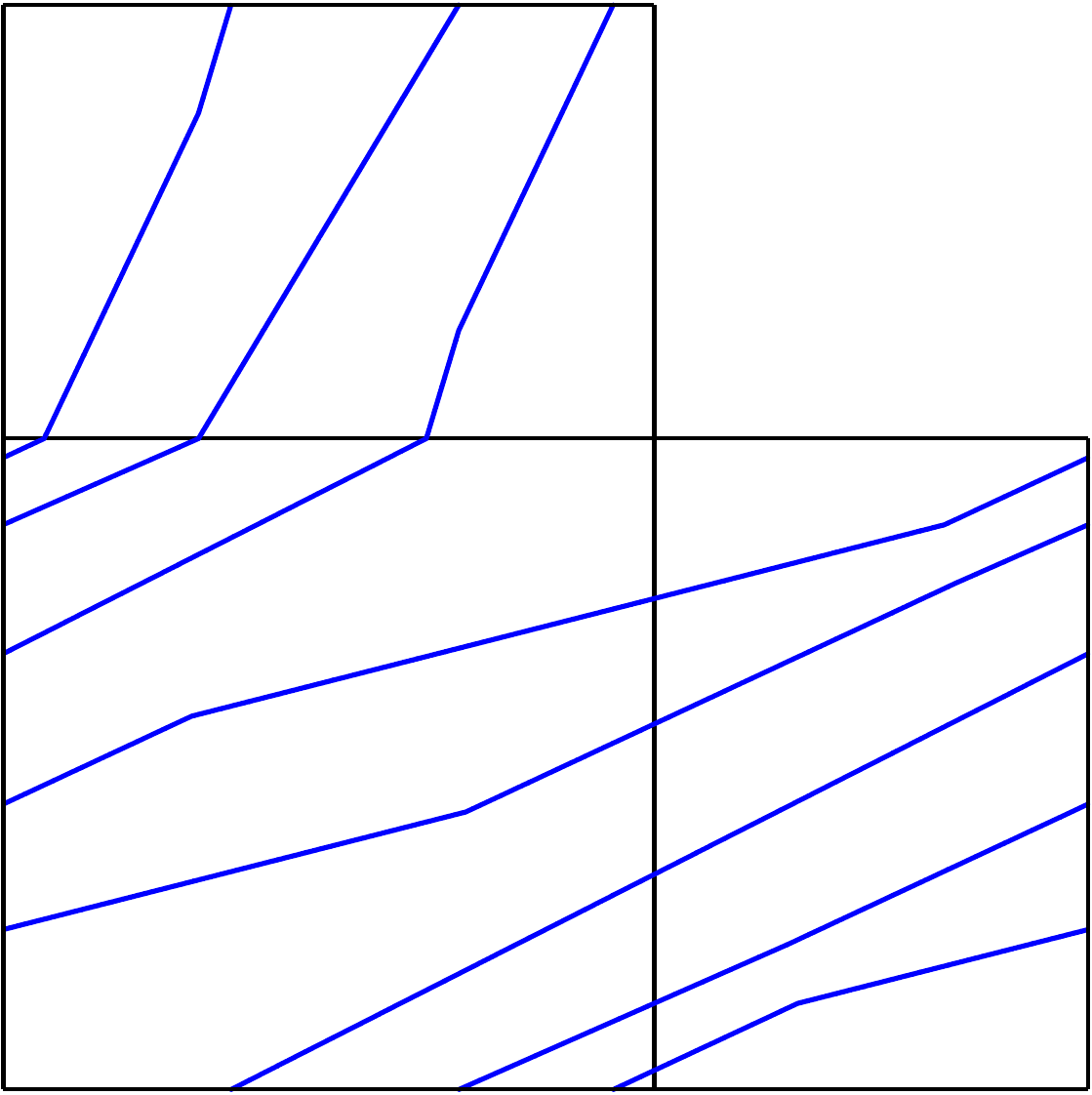}}}\hspace{1em}
\subfigure[\,co-rotating]{\label{fig:ttmlinecort}
  {\includegraphics[width=0.25\columnwidth]{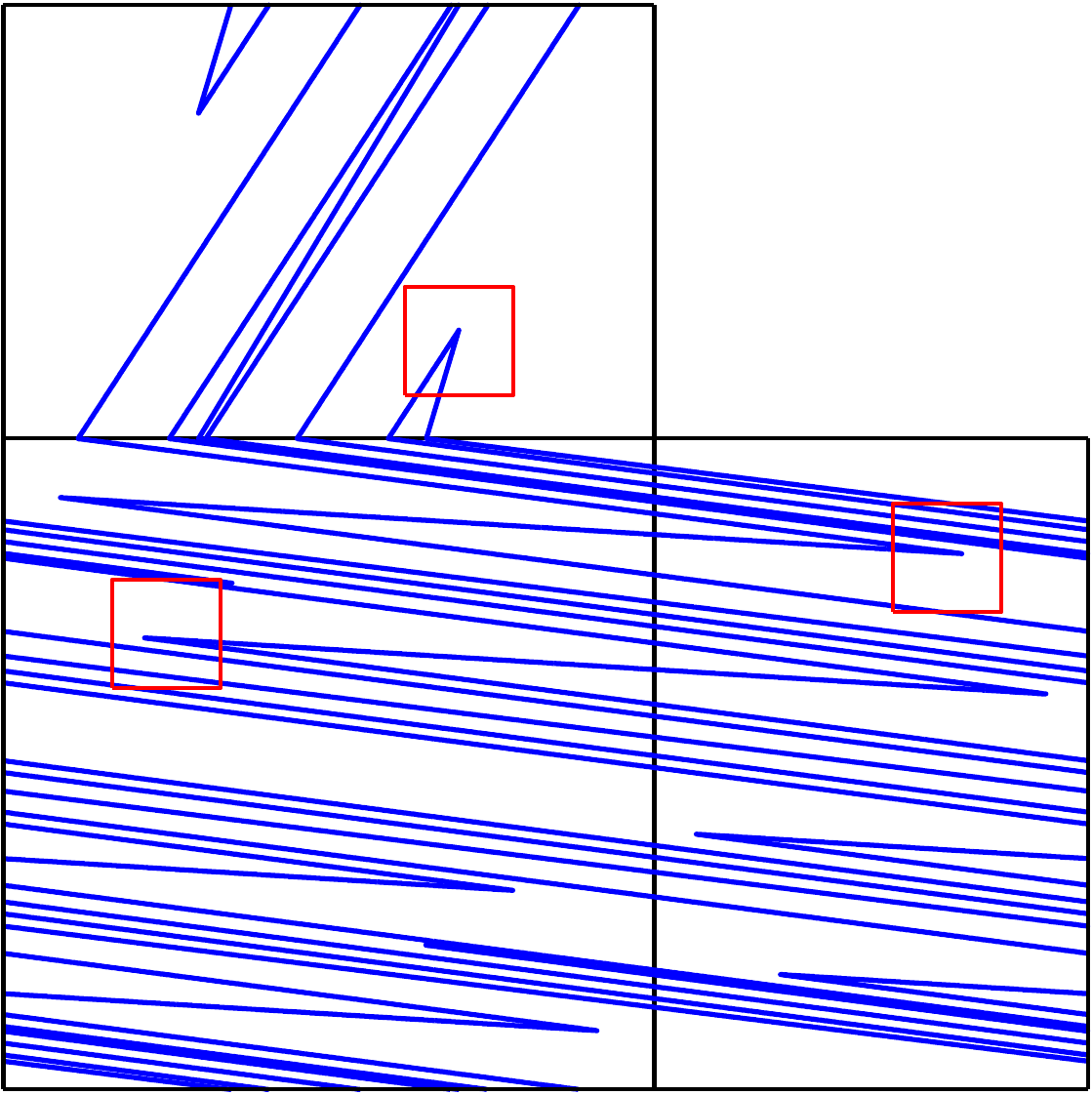}}}
\caption{Images of the same initial line segment after two iterations for (a) a
  counter-rotating toral LTM ($k=\ell=1$), and (b) a co-rotating
  toral LTM ($k=1$, $\ell=5$).  The small boxes highlight examples of
  secondary folding (acute angles) in the co-rotating case.}
\label{fig:ttmline}
\end{figure}

\subsection{Secondary Folding in Toral LTMs}

Recall that the toral LTM is the composition of a shear in a vertical
strip of the torus and a shear in a horizontal strip on the torus.  As
a result, the domain is partitioned into three regions: points that
undergo the vertical shear, but not the horizontal; points that
undergo the horizontal shear, but not the vertical; and points that
undergo both shears.  We label these three regions $R_{\V}$, $R_{\H}$,
and $R_{\L}$ as follows (see Figure \ref{fig:toralLTMregions}):
\begin{subequations}
\begin{align}
R_{\V} &= \{(x,y) \bmod 1 \st 0 \leq x \leq \alpha, \; \beta <
  y+xk\alphainv < 1\}; \\
R_{\H} &= \{(x,y) \bmod 1 \st \alpha < x < 1, \; 0 \leq y \leq \beta\}; \\
R_{\L} &= \{(x,y) \bmod 1 \st 0 \leq x \leq \alpha, \; 0 \leq
  y+xk\alphainv \leq \beta\},
\end{align}
\label{eq:allReqs}%
\end{subequations}
and also define $R = R_{\V} \cup R_{\H} \cup R_{\L}$.
\begin{figure}
\begin{center}
\includegraphics[width=0.3\columnwidth]{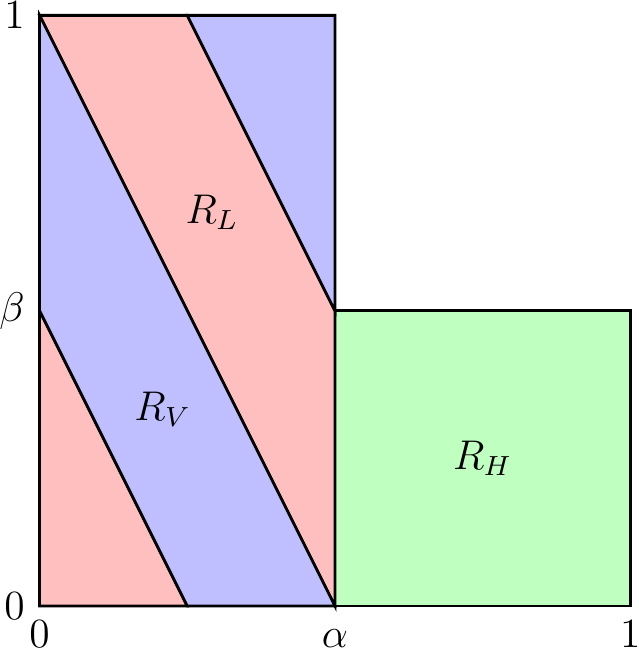} 
\end{center}
\caption{The three regions given by Equation (\ref{eq:allReqs}), shown here with $k=1$.}
\label{fig:toralLTMregions}
\end{figure} %
More specifically, for a point $z =(x,y)^\text{T}$, the toral LTM map is given by
\begin{equation}
\M(z) = \begin{cases} 
V z \bmod{1}, & \text{if $z \in R_V$}; \\
H z \bmod{1}, & \text{if $z \in R_H$}; \\
L z \bmod{1}, & \text{if $z \in R_L$},
\end{cases}
\label{eq:threemaps}
\end{equation}
where~$\V$,~$\H$, and~$\L$ are the matrices
\begin{equation}
\V = \begin{bmatrix} 1 & 0 \\ \kappa & 1 \end{bmatrix}, \quad
\H = \begin{bmatrix} 1 & \lambda \\ 0 & 1 \end{bmatrix}, \quad
\L = \H\V = \begin{bmatrix} 1 + \kappa\lambda & \lambda \\ \kappa & 1 \end{bmatrix},
\end{equation}
with $\kappa = k\alphainv$ and $\lambda = \ell\betainv$.  The toral
LTM is continuous and piecewise linear.

Because of the piecewise nature of linear toral LTMs, the image of a
line segment that crosses from one region to another will be a
segmented line.  This is the origin of all bends in iterates of line
segments.  Bends can be either obtuse or acute.  We conjecture that
obtuse bends do not contribute to the topological entropy, and so we
ignore them; however, acute bends (also called \emph{kinks}) always
appear associated with an increased topological entropy.  These kinks
are the manifestation of secondary folds on the torus.  To show the
existence (or non-existence) of kinks, we examine the unstable
manifold.  This is sufficient since, under repeated iteration, a line
segment will align with the unstable manifold and stretch in that
direction.

A subtle but important point here is that there is a set of
singularities, $\mathcal{S} \in R$, where the unstable (and stable)
manifolds don't exist~\cite{SturmanEtAl2006, Wojtkowski1980}.  This
set consists of all points on the boundaries of the $R_L$, $R_V$, and
$R_H$ regions, as well as all (forward and backward) images of these
points.  These singularities exist because the toral LTM is
continuous, but not differentiable, on the boundaries of the regions,
so the tangent vector is not well-defined at those points.  (These
singular points eventually become the vertices of bends in line
segments.)  Fortunately, $\mathcal{S}$ is a set of measure
zero~\cite{Wojtkowski1980}, and tangent vectors are well defined in
the neighborhood of the singularities.  So, using continuity, we will
speak of having an `unstable' manifold at each singular point, even
though it may be the vertex of a bend.  Then, since line segments tend
to align with the unstable manifold, if a system has a kinked unstable
manifold then any line segment will also develop kinks.

The following lemma gives the slope of the unstable manifold at any
(non-singular) point.

\begin{lemma}
\label{thm:contfracslope}
For a toral LTM, the slope of the unstable manifold at a
  point $z\in R \setminus \mathcal{S}$ is given by
\begin{equation}
  S_u(z) = \cfrac{1}{\lambda n_1
          + \cfrac{1}{\kappa m_1
          + \cfrac{1}{\lambda n_2
	  + \cfrac{1}{\kappa m_2
            + \dots}}}}
\label{eq:contfrac_unstable}
\end{equation}
where $m_i,n_i > 0$ if $z$ is in the horizontal strip, and $m_i,n_i
>0$ but $n_1=0$ otherwise.  The $m_i$ and $n_i$ are unique integers
that result from iterating $z$ backwards along its orbit.
\end{lemma}
\noindent
(Note that this continued fraction was derived in a slightly different
manner in~\cite{Wojtkowski1980}.)
\begin{proof}
We use the fact that under repeated forward iterations, all line
segments converge to the unstable
manifold.  To approximate the unstable manifold at point $z$, we first
follow $z$ backwards along its orbit to some point $z'$.  Then, we can
take (almost) any vector in the tangent space, $T_{z'}R =
\mathbb{R}^2$, and map it forward to $z$.  The image of the vector at
$z$ will be approximately tangent to the unstable manifold at $z$.
(This fails only when the initial vector is exactly aligned with the
stable manifold.)

Each iteration of the inverse linked twist map, $M^{-1}(z)$, leads
multiplication by either $V^{-1}$, $H^{-1}$, or $L^{-1}=V^{-1}H^{-1}$,
depending on the location of $z$.  So repeated backward iterations can
be written as
\begin{equation}
z' = M^{-N}(z) = V^{-m_j}H^{-n_j} \dots V^{-m_1}H^{-n_1}(z)
\end{equation}
where $N = \sum_{i=1}^j (m_i+n_i-1)$.  (The value of $j$ depends on
both $N$ and $z$.)  In general, $m_i,n_i > 0$.  However, for some
points, the first action of the inverse map is $V^{-1}$, not $H^{-1}$,
so for these points, $n_1=0$.  These are exactly the points that are
not in the horizontal strip.  (Additionally, for a given $N$ and $z$,
we might have $m_j=0$.  In this case, (\ref{eq:contfrac_finite}) below
ends with~$\lambda n_j$; this does not change the argument.)

Let $w'$ be a vector in $T_{z'}R = \mathbb{R}^2$.  Since $DV=V$ and $DH=H$, $w$ and $z'$ evolve according to exactly the same matrix multiplication.  That is, $w \in T_zR$ can be written as 
\begin{equation}
w = H^{n_1}V^{m_1}H^{n_2}V^{m_2}\dots H^{n_j}V^{m_j}w'\,.
\end{equation}
Now consider how multiplication by $H^nV^m$ affects the slope of a
vector.  If the vector initially has slope $s$, then after
multiplication by $H^nV^m$ it has slope
\begin{equation}
s' = \frac{s + \kappa m}{1 + \lambda n(s+\kappa m)} 
   = \cfrac{1}{\lambda n + \cfrac{1}{\kappa m + s}} \; .
\end{equation}
Repeated iteration to find the slope at $w$ gives the finite continued fraction
\begin{equation}
\label{eq:contfrac_finite}
 \text{slope}(w) = \cfrac{1}{\lambda n_1
          + \cfrac{1}{\kappa m_1
          + \cfrac{1}{\lambda n_2
	  + \cfrac{1}{\kappa m_2
          + \cfrac{1}{\dots
          + \cfrac{1}{\kappa m_j}}}}}} \; .
\end{equation}

To find the exact slope, let $N~\rightarrow~\infty$.  This means
$j~\rightarrow~\infty$ and the finite continued fraction becomes the
infinite continued fraction given in (\ref{eq:contfrac_unstable}).  So
it only remains to show that the continued fraction converges.  In the
counter-rotating case, $\kappa\lambda>0$, and we have $\kappa m_i \geq
1$ and $\lambda n_i\geq1$.  Hence $\sum(\kappa m_i + \lambda n_i)$
diverges, and the continued fraction converges to a finite value by
the Seidel--Stern Theorem~\cite{JonesThron1980}.

By definition, the co-rotating toral LTM must have $k\ell< 0$, or
equivalently $\kappa\lambda<0$.  However, we add another restriction
and require $\kappa\lambda < -4$.  This condition ensures that
matrix~$\L$ is hyperbolic and, consequently, that the toral LTM has an
ergodic partition~\cite{Wojtkowski1980}.  (With a slightly stronger
condition, the toral LTM will also be Bernoulli~\cite{Przytycki1983}.)
To show convergence in this case, we transform the continued fraction
into the form
\begin{equation}
  S_u(z) = \cfrac{(\lambda n_1)^{-1}}{1
          + \cfrac{(\kappa\lambda m_1n_1)^{-1}}{1
          + \cfrac{(\kappa\lambda m_1n_2)^{-1}}{1
	  + \cfrac{(\kappa\lambda m_2n_2)^{-1}}{1
            + \dots}}}}\ .
\label{eq:contfrac_alt}
\end{equation}
Note that since $\kappa\lambda <-4$, each of the numerators (except
for the first) has magnitude less than $1/4$.  The continued fraction
then converges to a finite value by Worpitzky's
Theorem~\cite{JonesThron1980}.  (The case where $n_1=0$ follows
similarly.)
\end{proof}

We intend to show that counter-rotating toral LTMs do not have kinks, while co-rotating toral LTMs do.  Therefore, we proceed by investigating each type separately.  For convenience, we assume in both cases that $k>0$.

\subsubsection{Counter-rotating}

Since we are only considering maps with $k>0$, we have $\ell>0$ as well.  Then every part of the continued fraction (\ref{eq:contfrac_unstable}) is positive, and we conclude that the slope of the unstable manifold, $S_u$, is positive at every point $z \in R \setminus \mathcal{S}$.  Furthermore, it is easy to show that the orientation of the unstable manifold is preserved under the action of the map.

Define two cones in tangent space $TR=\mathbb{R}^2$ by
\begin{align}
C_1 &= \{(u,v) \st u>0, v>0  \}, \\
C_3 &= \{(u,v) \st u<0, v<0  \}.
\end{align} 
(The subscripts refer to the standard quadrant numbering.)  Note that the definitions are independent of $z$.

\begin{lemma}
\label{thm:ctrt_invcones}
For $k,\ell>0$, the cones $C_1$ and $C_3$ are invariant under both the vertical twist ($V$) and the horizontal twist ($H$).
\end{lemma}
\begin{proof}
Take a vector $w=(u,v)$.  Then $\V w = \left( u, v+u\kappa \right)$, and $\H w = \left( u + v\lambda, v \right)$.  Clearly, if $w \in C_1$ then both $\V w$ and $\H w$  are also in $C_1$ (and therefore $\L w \in C_1$).  Similarly, if $w \in C_3$ then both $\V w$ and $\H w$ are also in $C_3$.
\end{proof}

\begin{corollary}
\label{thm:ctrt_nokinks}
A line segment with positive slope will never kink under applications of $M$.
\end{corollary}
\begin{proof}
Consider a line segment, $\gamma$, with positive slope passing through a point $z$.  Let $w_1$ and $w_3$ be vectors in $T_zR$ that are tangent to $\gamma$ and such that $w_1 \in C_1$ and $w_3 \in C_3$.  Then from Lemma \ref{thm:ctrt_invcones}, $DM(w_1) \in C_1$ and $DM(w_3) \in C_3$ regardless of which region $z$ is in.  Consequently, the angle between the vectors is obtuse, and no kink has formed.
\end{proof}

\begin{theorem}
The unstable manifold of a counter-rotating toral LTM has no kinks.
\end{theorem}
\begin{proof}
  Suppose that there is a kink in the unstable manifold.  Then there
  must be a portion of the unstable manifold that is a straight line,
  crosses a region boundary, and maps to the kink.  But $S_u > 0$ for
  counter-rotating toral LTMs, which contradicts Corollary
  \ref{thm:ctrt_nokinks}.  Hence, there are no kinks in the unstable
  manifold.
\end{proof}

\subsubsection{Co-rotating}
\label{sec:co-rotating}

Next we consider co-rotating toral LTM, and again restrict attention
to ones with $k>0$ and $\kappa\lambda < -4$, so that the matrix $L$ is
hyperbolic.  For such toral LTMs, we prove that the unstable manifold
is kinked.

We begin by using the eigenvectors of $L$ to define two
pairwise-invariant cones (as in~\cite{SturmanEtAl2006}).  For each
point $z\in R$, we define two cones in the tangent space $T_zR$ by
\begin{align}
C(z) &= \{(u,v) \st m_* \leq v/u \leq 0  \}, \\
\tild{C}(z) &= \{(u,v) \st 0 < m_* + \kappa \leq v/u  \} ,
\end{align} 
where $m_*$ is the slope of the expanding eigenvector of $L$ and for
co-rotating maps:
\begin{equation}
m_* =
\frac{2\kappa}{\kappa\lambda-\sqrt{\kappa\lambda(\kappa\lambda+4)}}
=
-\frac{\kappa\lambda+\sqrt{\kappa\lambda(\kappa\lambda+4)}}{2\lambda}
< 0.
\end{equation}
Note that cone definitions are independent of $z$, so we can refer to them simply as $C$ and $\tild{C}$.  For convenience, we let $I_C = [m_*,0]$ and $I_{\tild{C}} = [m_* + \kappa,\infty)$ denote the range of slopes in cones $C$ and $\tild{C}$ respectively.  The following lemma shows that tangent vectors to the unstable manifold lie precisely in these cones.

\begin{lemma}
\label{thm:unst_angle_cort}
For $\kappa\lambda<-4$ and $z \in R \setminus \mathcal{S}$, the slope of the unstable manifold, $S_u(z)$, lies in $C$ if $z$ is in the horizontal strip, and lies in $\tild{C}$ otherwise.
\end{lemma}
\begin{proof}
This follows from a modification of a theorem by Hillam and
Thron regarding convergence regions of continued fractions~\cite{HillamThron1965}.  They look at a sequence of
maps, $F_i = f_1 \circ f_2 \circ \cdots \circ f_i$, where $f_i$ has
the form~$f_i(\tau) = {a_i}/{(b_i + \tau)}$, and the sequence
$\{F_i(0)\}$ gives successive convergents of a continued fraction.
They prove that if there exists a disk $D = \{w\in \mathbb{C} \st |w-c| \leq r\}$, with $|c| < r$ and $f_i(D)\subseteq D$, then the
continued fraction converges to a value in $D$.

We modify this proof for our purposes so that $f_i(\tau)$ has the form
\begin{equation}
f_i(\tau) = \cfrac{1}{\lambda n_i + \cfrac{1}{\kappa m_i + \tau}}\,.
\end{equation}
Now the sequence $\{F_i(0)\}$ gives the even numbered convergents of
the continued fraction (\ref{eq:contfrac_unstable}).  Since we know
that the continued fraction converges (Lemma \ref{thm:contfracslope}),
 these even numbered convergents will converge to the same
value as the continued fraction.  

The disk we use is $D = \{w\in\mathbb{C} \st |w-m_*/2| \leq -m_*/2\}$.  This is the smallest disk possible, because when $n_i=m_i=1$, $f_i(m_*)=m_*$.  To see that $f_i(D) \subseteq D$, re-write $f_i(\tau)$ as
\begin{equation}
f_i(\tau) = \frac{\tau + \kappa n_i}{\lambda n_i \tau + \kappa \lambda n_i m_i + 1}
\end{equation}
and view it as a M\"{o}bius transformation.  Then $f_i(\partial D)$ is a circle with center and radius given by:
\begin{align}
\text{center} &= \frac{1}{2}\left( \frac{\kappa m_i}{1+\kappa \lambda m_i n_i} + \frac{m_* + \kappa m_i}{1 + \lambda n_i (m_* + \kappa m_i)} \right) \\
\text{radius} &= \frac{m_*}{2(1 + \kappa \lambda m_i n_i) |1 + \lambda n_i (m_* + \kappa m_i)|}.
\end{align}
From these, it is easy to check that $f_i(D)\subseteq D$.  The intersection of $D$
with the real axis is the interval $I_C = [m_*,0]$, i.e.\ exactly the
range of slopes in $C$. Thus, for $n_i>0$, $S_u(z) \in I_C$ and the
unstable manifold is in $C$.

When $n_1=0$, equation (\ref{eq:contfrac_unstable}) has the form 
\begin{equation}
  S_u(z) = \kappa m_1 + \cfrac{1}{\lambda n_2
          + \cfrac{1}{\kappa m_2
          + \cfrac{1}{\lambda n_3
            + \dots}}}
\end{equation}
and $S_u(z) \in I_{\tild{C}}$ follows in the same manner.
\end{proof}

The following lemma shows how a kink may be formed after one application of the toral LTM.
\begin{lemma}
\label{thm:kinkzone}
Take a point $z \in \mathcal{S} \subset R$ on the boundary between the
$R_{\L}$ and $R_{\H}$ regions (i.e. $z \in R_{L} \cap
\overline{R_{H}}$).  Then a line segment passing through $z$ (possibly
with a bend at $z$) and with the segments on either side of $z$ having
slope $s\in(-\kappa-\lambda^{-1},0]$ will form an acute angle, with
vertex at $M(z)$, after one application of $M$.
\end{lemma}
\begin{proof}
Suppose for simplicity that $z$ is on the right-hand boundary of $R_L$, as in Figure~\ref{fig:kink}.  (The case where $z$ is on the left-hand boundary is proved in an identical manner.)
\begin{figure}
\begin{center}
\includegraphics[width=.7\columnwidth]{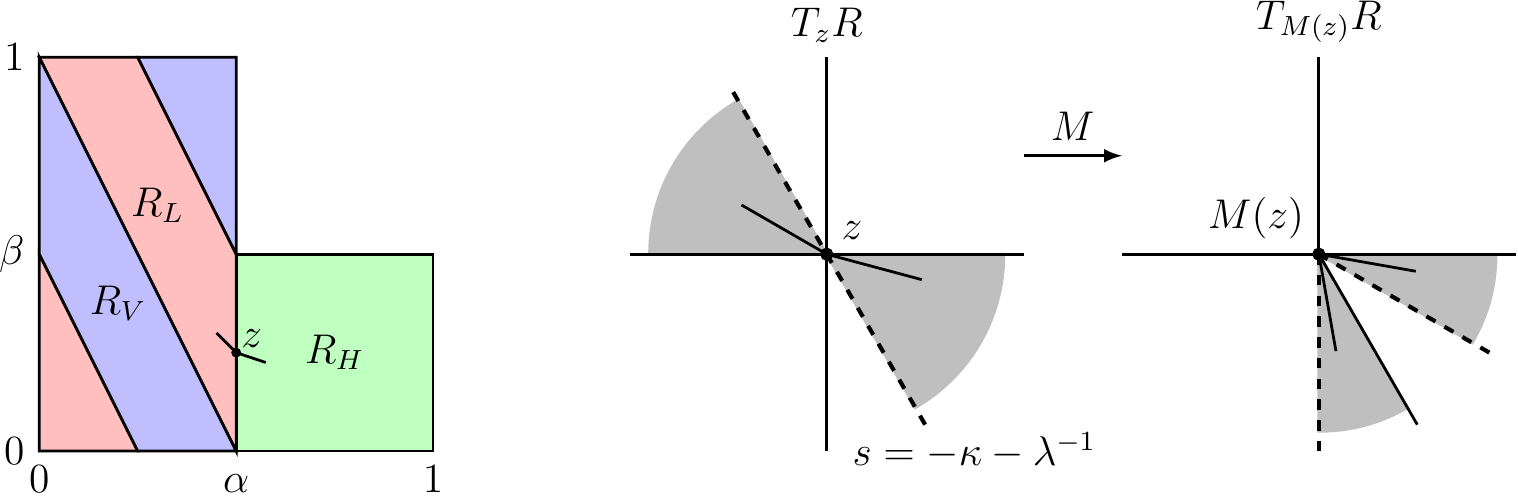} 
\end{center}
\caption{A (possibly jointed) line segment passing through a point on the boundary of
  $R_{\L}$ and $R_{\H}$, with slope $m\in (-\kappa-\lambda^{-1},0]$, will form a kink after
  one iteration of $\M$.}
\label{fig:kink}
\end{figure}
The portion of the line segment that lies to the left of point $z$ is
in the region $R_{\L}$.  Let $u=(u_1,u_2)$ be a vector that is
parallel to the line segment and points left.  That is, $u_1<0$ and,
because $u_2/u_1\in(-\kappa-\lambda^{-1},0]$, we have $u_2 \geq 0$.
In this region the action of the toral LTM is multiplication by the matrix $\L = \H\V$. So if the segment left of $z$ is initially parallel to $u$, then it will be parallel to $u'=\L u$ after $M$ has been applied.  We can compute $u' = \left(u_1 + u_1\kappa\lambda + u_2\lambda, u_2 + u_1\kappa \right)$,
and it is easy to see that $u'$ points strictly into the 4th quadrant.

Meanwhile, the portion of the line segment that lies to the right of
point $z$ is in the region~$R_\H$.  Let $v = (v_1,v_2)$ be a vector
that is parallel to the line segment and points right.  That is,
$v_1>0$ and $v_2 \leq 0$.  Since we are in region $R_\H$, the action
of the toral LTM is multiplication by the matrix $\H$.  So after $M$
is applied, the line segment will be parallel to $v' = \H v$.  Simple
computation yields $v' = \left(v_1 + v_2\lambda, v_2 \right)$, and we
can see that $v'$ points either horizontally or into the 4th quadrant.
Thus, $M(z)$ is the vertex of an acute angle and a kink has formed.
\end{proof}

\begin{theorem}
The unstable manifold of a co-rotating toral LTM with $\kappa\lambda<-4$ has kinks.
\end{theorem}
\begin{proof}
Lemma \ref{thm:unst_angle_cort} shows that for non-singular points in
the horizontal strip, the unstable manifold is in cone $C$.  Since
$I_C \subset (-\kappa-\lambda^{-1},0]$, pieces of the unstable
manifold that cross the boundaries between $R_L$ and $R_H$ satisfy the hypotheses
of Lemma \ref{thm:kinkzone}.  Thus, these pieces kink after one
application of $M$.  Since pieces of unstable manifold map to other
pieces of unstable manifold, this implies that there already existed
kinks in the unstable manifold.
\end{proof}

\section{Discussion}
\label{sec:discussion}

We set out to explain the gap between the homological lower bound on
the topological entropy and its numerically-observed value.  We
observed that, in both fluid-dynamical systems and toral LTMs, the
presence of a gap appeared associated with `secondary folding' -- the
presence of extra folds in material lines, not associated with
topological obstacles such as rods.  Though we have not been able to
rigorously show that the gap is due to the folds, we were able to show
that counter-rotating toral LTMs never have folds, while hyperbolic
co-rotating toral LTMs always have folds.  This correlates perfectly
with the presence or absence of a gap in the topological entropy.
Future work will aim to show that secondary folding is the direct
cause of the extra entropy.


\section*{Acknowledgments}

The authors thank Phil Boyland and Rob Sturman for many enlightening
discussions.  This work was funded by the Division of Mathematical
Sciences of the US National Science Foundation, under grant
DMS-0806821.

\bibliography{setbib}

\end{document}